\documentclass[hidelinks,12pt,a4paper,reqno]{amsart}%{article}%The reqno document class option displays equation numbers on the right
\usepackage[utf8]{inputenc}
\usepackage{enumitem}
\usepackage{amsmath}
\usepackage{amsfonts}
\usepackage{amssymb}
\usepackage[all]{xy}
\usepackage{xfrac}
\usepackage[yyyymmdd,hhmmss]{datetime}

\usepackage[english]{babel}
\usepackage{dsfont}
\usepackage{bbold}
\usepackage{soul}

\usepackage[backref=page]{hyperref}
\renewcommand*{\backref}[1]{}
\renewcommand*{\backrefalt}[4]{%
    \ifcase #1 (Not cited.)%
    \or        (Cited on page~#2.)%
    \else      (Cited on pages~#2.)%
    \fi}

\newtheorem{theorem}{Theorem}[section]

\newtheorem{definition}{Definition}[section]
\newtheorem{proposition}[theorem]{Proposition}

\newtheorem{example}{Example}[section]

 \author{N. Martins-Ferreira}
 \address{School of Technology and Management, Centre for Rapid and Sustainable Product Development - CDRSP, Polytechnic Institute of Leiria, P-2411-901 Leiria, Portugal.}
 \email{martins.ferreira@ipleiria.pt}
\title[]{On the structure of topological spaces}


\thanks{
This work is supported by the Fundação para a Ciência e a Tecnologia (FCT) and Centro2020 through the Project references: UID/Multi/04044/2019; PAMI - ROTEIRO/0328/2013 (Nº 022158); Next.parts (17963), and also by CDRSP and ESTG from the Polytechnic of Leiria.
}

\begin{document}

\begin{abstract}
The structure of topological spaces is analysed here through the lenses of fibrous preorders. Each topological space has an associated fibrous preorder and those fibrous preorders which return a topological space are called spacial. A special class of spacial fibrous preorders consisting of an interconnected family of preorders indexed by a unitary magma is called cartesian and studied here. Topological spaces that are obtained from those fibrous preorders, with a unitary magma \emph{I}, are called \emph{I}-cartesian and characterized. The characterization reveals a hidden structure of such spaces. Several other characterizations are obtained and special attention is drawn to the case of a monoid equipped with a topology. A wide range of examples is provided, as well as general procedures to obtain topologies from other data types such as groups and their actions. Metric spaces and normed spaces are considered as well.
\end{abstract}

\maketitle

\today; \currenttime

\section{Introduction}

The definition of a topological space as we know it today has a long history (see e.g. \cite{James 1999}) and it is so rich and full of twists and turns that the simple task of tracking down its origins is transformed into an overwelming undertaking (see e.g. \cite{Moore 2008}). Arguably, it starts at the beginning of the twentieth century with seminal works notably by Dedekind \cite{Dedekind 1931},  Lebesgue \cite{Lebesgue 1902}, Riesz \cite{Riesz 1905}, de la Vallé Poussin \cite{Poussin 1916} and Frechét \cite{Frechet 1906, Frechet 1921}, whose primary interests were still focused on generalizing results from the previous century. The first few decades were characterized by several further improvements and alternative definitions, notably by Hausdorff \cite{Hausdorff 1914, Hausdorff 1927}, Carathéodory \cite{Caratheodory 1918}, Kuratowski \cite{Kuratowski 1922}, Tietze \cite{Tietze 1923}, Aleksandrov \cite{Aleksandrov 1925}, while the vision of which definitions would be adopted as primitive and which concepts aught to be derived was nothing but a blur. It was only in the mid thirties that the works by Aleksandrov and Hopf \cite{Aleksandrov and Hopf 1935}, Sierpinski \cite{Sierpinski 1934}, Kuratowski \cite{Kuratowski 1933} and Lefschetz \cite{Lefschetz 1930, Lefschetz 1942} started to be influential. The current established notion has finally settled down with the wide dissemination of the classical works by Bourbaki \cite{Bourbaki 1951} and kelly \cite{Kelley 1955}. From this modern point of view, a topological space is presented as a pair $(X,\tau)$, where $X$ is a set and $\tau\subseteq \mathcal{P}(X)$ is a topology on $X$, i.e., a collection of so-called open sets which are nothing but subsets of $X$, closed under finite intersections and arbitrary unions and moreover at least the empty set and the set $X$ itself must be open.

In spite of all the advantages of such an abstract definition, which has unquestionably led to great progresses over the last more than hundred years not only in mathematics but also in physics and in other areas of knowledge, there are still a few difficulties in the treatment of topological spaces at this level of abstraction. The trouble is that the process of simplifying the definition of topological space to its bare bone has two undesirable effects: one is the appearance of redundant information, another one is that its true structure is being hidden. Let us mention two concrete examples of this phenomena.

If $X$ is a finite set then, as it is well known \cite{Aleksandrov 1937}, a topology on $X$ is nothing but a preorder, which is simply a reflexive and transitive relation. Shouldn't this be a simple observation that would follow from the definition of a topological space?

If $X$ is a group and if the map $X\times X\to X$; $(x,y)\mapsto xy^{-1}$ is required to be continuous (with the product topology on its domain and the group operation used to form $xy^{-1}$) then we get a topological group \cite{Montgomery 1955} and  the range of possible topologies is severely restricted. Surprisingly, a topological group is presented as a group equipped with an arbitrary topology (which is required to be compatible with the group operation) while it is clear that such topologies must be simpler than arbitrary ones. This is mainly because the simplification on the structure of those topologies is not apparent from its definition.  Shouldn't there be a way of identifying which key features of an arbitrary topology gives compatibility with a group operation?

Some attempts have been made to overcome these difficulties, notably Brown in his book Topology and Groupoids \cite{Brown 1988}. In \cite{NMF 2014} the notion of spacial fibrous preorder was used to structure the category of topological spaces so that it becomes apparent that every preorder gives rise to a topology and moreover, for finite sets, there are no possibilities other than that. 

The purpose of this paper is to use some ideas from \cite{NMF 2014} in a better understanding on the structure of metric spaces and topological groups from the point of view of abstract topological spaces.

 A spacial fibrous preorder  \cite{NMF 2014} has the advantage of being intuitive (it is in some sense a modification of a preorder) and yet there is a categorical equivalence between spacial fibrous preorders and topological spaces. Moreover, it suggests that the study of continuous maps may be pursued as a fine structure rather than a property. Instead of a map having the property of being continuous or not it may rather be viewed as a fine structure which can detect different levels of continuity such as uniform continuity and other related concepts. Nevertheless, a reader who is not familiar with spacial fibrous preorders in the first place may fail to consider it as an appealing structure. For that reason we have decided to present here a simpler version which is nonetheless still sufficient for our purposes. In order to distinguish it from the general case we will call it \emph{cartesian spacial fibrous preorder}.
 
  A \emph{fibrous preorder} \cite{NMF 2014} is a sequence $R\to A\to B$ with $R\subseteq A\times B$ satisfying some conditions. If the set $A$ is the cartesian product of a set $I$ and the set $B$ then we will speak of a \emph{cartesian fibrous preorder}. Not every fibrous preorder is realizable as a topological space, only the \emph{spacial} ones are so \cite{NMF 2014}. In the same manner we will restrict cartesian fibrous preorders to spacial ones and simplify its structure a little bit by considering a unitary magma as its indexing set, $I$, rather than the slightly more general structure considered in \cite{NMF 2014}. Note that when $I=\{1\}$ is a singleton set then $A=I\times B$ can be identified with $B$ and the relation $R\subseteq B\times B$ becomes precisely a preorder.

\section{Cartesian spacial fibrous preorders and unitary magmas}

Let $(I,\cdot,1)$ be a unitary magma. That is a set $I$ together with a distinguished element $1\in I$ and a binary operation $${I\times I\to I};\quad (i,j)\mapsto i\cdot j,$$ such that $i\cdot 1=i=1\cdot i$ for all $i\in I$. Sometimes we write $i\cdot j$ simply as $ij$. A unitary magma is the same as a monoid when the associativity condition  $i(jk)=(ij)k$ holds true for all $i,j,k\in I$.  

\begin{definition}\label{def: cartesian spacial fibrous preorder}
A \emph{cartesian spacial fibrous preorder}, indexed by the unitary magma $(I,\cdot,1)$, is a system $(X,(\leq^{i})_{i\in I},(\partial^{i})_{i\in I})$ where $X$ is a set and for every $i\in I$, $\leq^{i}$ is a binary relation on $X$ whereas $\partial^{i}$ is a partial map $X\times X\to I$ which is defined for all pairs $(x,y)$ such that $x\leq^{i} y$. Moreover, the following conditions must hold for all $x,y,z\in X$ and $i,j\in I$:
\begin{enumerate}
\item[(C1)] $x\leq^{i} x$;
\item[(C2)] if $x\leq^{i} y$, $\partial^{i}(x,y)=j$ and $y\leq^{j} z$ then $x\leq^{i} z$;
\item[(C3)] if $x\leq^{ij} y$ then $x\leq^{i} y$ and $x\leq^{j} y$.
\end{enumerate}
\end{definition}

A metric space is the best example to illustrate the structure while providing useful intuition further on.

\begin{example}
 Let $I=\mathbb{N}$ be the unitary magma of natural numbers with $1\in \mathbb{N}$ as the neutral element and the usual multiplication as binary operation. Let $X$ be any metric space with metric $d\colon{X\times X\to [0,+\infty[}$. Under these assumptions we put $x\leq^n y$ if and only if $d(x,y)<\frac{1}{n}$ and choose $\partial^{n}(x,y)=k$ to be such that $\frac{1}{k}\leq \frac{1}{n}-d(x,y)$, with $x,y\in X$ and $n,k\in \mathbb{N}$. It is not difficult to see that conditions (C1)-(C3) are satisfied.
\end{example}

As for metric spaces, cartesian spacial fibrous preorders give rise to topological spaces.

%An equivalence betwen spacial fibrous preorders and topological spaces has been proven in \cite{NMF 2014}. From there we can deduce that every cartesian spacial fibrous preorder $(X,(\leq^{i})_{i\in I},(\partial^{i})_{i\in I})$ gives rise to a topological space $(X,\tau)$ with $\tau$ defined as 
%\begin{equation}\label{eq: tau}
%\mathcal{O}\in \tau \Leftrightarrow \forall x\in \mathcal{O}, \exists i\in I, N(i,x)\subseteq \mathcal{O}
%\end{equation}
%where $N(i,x)=\{y\in X\mid x\leq^{i} y\}$. Furthermore, a topological space $(X,\tau)$ is obtained from a cartesian spacial fibrous preorder indexed by a unitary magma $(I,\cdot,1)$ if and only if there exist maps $N\colon{I\times X\to \tau}$ and $\gamma\colon{\{(U,x)\mid x\in U\in \tau\}\to I}$ such that $x\in N(i,x)$ for every $i\in I$ and $x\in X$, and $N(\gamma(U,x),x)\subseteq U$, for all $x\in U\in \tau$. Under these assumptions it is not difficult to see that the system  $(X,(\leq^{i})_{i\in I},(\partial^{i})_{i\in I})$  defined as $x\leq^{i}y$ if and only if $y\in N(i,x)$ and $\partial^{i}(x,y)=\gamma(N(i,x),y)$ is a cartesian spacial fibrous preorder indexed by the unitary magma $I$. This result may be summarized as follows.

\begin{proposition}\label{thm: fibrous to spaces}
Every cartesian spacial fibrous preorder $$(X,(\leq^{i})_{i\in I},(\partial^{i})_{i\in I})$$ gives rise to a topological space $(X,\tau)$ with $\tau$ defined as 
\begin{equation}\label{eq: tau}
\mathcal{O}\in \tau \Leftrightarrow \forall x\in \mathcal{O}, \exists i\in I, N(i,x)\subseteq \mathcal{O}
\end{equation}
where $N(i,x)=\{y\in X\mid x\leq^{i} y\}$.
\end{proposition}
\begin{proof}
This is a special case of the equivalence betwen spacial fibrous preorders and topological spaces \cite{NMF 2014}. Having a cartesian spacial fibrous preorder $(X,(\leq^{i})_{i\in I},(\partial^{i})_{i\in I})$ we define a spacial fibrous preorder $(R,A,B,\partial,p,s,m)$ (see \cite{NMF 2014}) as:
\begin{eqnarray}
R\subseteq I\times X\times X,\quad  A=I\times X,\quad B=X
\end{eqnarray}
with $R$, $\partial \colon{R\to A}$, $p\colon{A\to B}$, $s\colon{B\to A}$ and $m\colon{A\times_B A \to A}$, defined as follows:
\begin{eqnarray*}
R=\{(i,x,y)\mid x\leq^{i} y\},\\
\partial(i,x,y)=(\partial^{i}(x,y),y),\\
p(i,x)=x,\\
s(x)=(1,x),\\
m(i,j,x)=(i\cdot j,x).
\end{eqnarray*}
Nevertheless, a direct proof is easily obtained by showing that $N(i,x)$ is a system of open neighbourhoods.
\end{proof}

It is clear that when a cartesian spacial fibrous preorder is obtained from a metric space then its induced topology is the same as the usual topology generated by the metric.

\section{A characterization of I-cartesian spaces}

The following result characterizes those topological spaces that are obtained from a cartesian spacial fibrous preorder. Such spaces will be called $I$-cartesian when $I$ is the  indexing unitary magma. Most of the time we will be considering an arbitrary unitary magma, $I$, however, it is useful from time to time to recall that our motivating example is the unitary magma of natural numbers. For that reason we will use the letters $i,j,k$ to represent elements in the set $I$ as well as the letters $n,m,k$.
%
%\begin{definition}\label{def: I-cartesian spaces}
%Let $(I,\cdot,1)$ be a unitary magma. A topological space $(X,\tau)$ is said to be \emph{$I$-cartesian} if there exist two maps $N\colon{I\times X\to \tau}$ and $\gamma\colon{\{(U,x)\mid x\in U\in \tau\}\to I}$ such that $x\in N(i,x)$ for every $i\in I$ and $x\in X$, and $N(\gamma(U,x),x)\subseteq U$, for all $x\in U\in \tau$.
%\end{definition}

\begin{theorem}\label{thm: I-cartesian spaces}
Let $(I,\cdot,1)$ be a unitary magma and $(X,\tau)$ a topological space. The following conditions are equivalent:
\begin{enumerate}
\item[(a)] The space $X$ is $I$-cartesian.
\item[(b)] The topology $\tau$ is determined by a cartesian spacial fibrous preorder $(X,(\leq^{i})_{i\in I},(\partial^{i})_{i\in I})$ as 
\begin{equation}\label{eq: tau (copy)}
\mathcal{O}\in \tau \Leftrightarrow \forall x\in \mathcal{O}, \exists i\in I, N(i,x)\subseteq \mathcal{O}
\end{equation}
with $N(i,x)=\{y\in X\mid x\leq^{i} y\}$.

\item[(c)] There exists a map $N\colon{I\times X\to\mathcal{P}(X)}$ such that
\begin{enumerate}
\item[(i)] $x\in N(n,x)$, for all $n\in I$, $x\in X$
\item[(ii)] $N(n,x)\subseteq\{y\in X\mid \exists k\in I, N(k,x)\subseteq N(n,x) \}$,
\item[(iii)] $N(n\cdot m,x)\subseteq N(n,x)\cap N(m,x)$, $n,m\in I$, $x\in X$ 
\end{enumerate}
for which $\tau$ is determined as \[\mathcal{O}\in \tau \Leftrightarrow \forall x\in \mathcal{O}, \exists n\in I, N(n,x)\subseteq \mathcal{O}.
\]
\item[(d)] There exists a ternary relation $R\subseteq I\times X\times X$ together with a map $p\colon{R\to I}$ such that
\begin{enumerate}
\item[(i)] $(n,x,x)\in R$, for all $n\in I$, $x\in X$
\item[(ii)] if $(n,x,y)\in R$ and $(p(n,x,y),y,z)\in R$ then $(n,x,z)\in R$
\item[(iii)] if $(n\cdot m,x,y)\in R$ then $(n,x,y)\in R$ and $(m,x,y)\in R$
\end{enumerate}
for which $\tau$ is determined as 
\begin{equation*} %\label{eq: tau (copy)2}
\mathcal{O}\in \tau \Leftrightarrow \forall x\in \mathcal{O}, \exists n\in I, N_R(n,x)\subseteq \mathcal{O}
\end{equation*}
 with $N_R(n,x)=\{y\in X\mid (n,x,y)\in R\}$.

\item[(e)] There are maps $\eta\colon{I\times X\to \tau}$ and $\gamma\colon{\{(U,x)\mid x\in U\in \tau\}\to I}$ such that:
\begin{enumerate}
\item[(i)] $x\in \eta(n,x)$ for every $n\in I$ and $x\in X$,
\item[(ii)] $\eta(\gamma(U,x),x)\subseteq U$, for all $x\in U\in \tau$,
\item[(iii)] $\eta(n\cdot m,x)\subseteq \eta(n,x)\cap \eta(m,x)$, for all $n,m\in I$, and $x\in X$.
\end{enumerate}
\end{enumerate}
\end{theorem}

\begin{proof}
Conditions $(a)$ and $(b)$ are equivalent by definition. 

To prove $(b)$ implies $(c)$ we start with a cartesian spacial fibrous preorder and define $N(n,x)=\{y\in X\mid x\leq^n y\}$. Conditions $(c)(i)$ and $(c)(iii)$ follow respectively from axioms $(C1)$ and $(C3)$. To prove $(c)(ii)$ we start with $y\in N(n,x)$, that is, $x\leq^n y$, and observe that there exists $k=\partial^n(x,y)\in I$ such that $N(k,y)\subseteq N(n,x)$. Indeed, if $z\in N(k,y)$, that is, $y\leq^k z$, then by $(C2)$ we have $x\leq^n z$, which is the same as saying $z\in N(n,x)$. This shows that the map $N\colon{I\times X\to\mathcal{P}(X)}$ satisfies $(c)(ii)$. It remains to show that $\tau$ is determined by it. This follows from Proposition \ref{thm: fibrous to spaces} and the assumption that $\tau$ is determined as in equation $(\ref{eq: tau (copy)})$.

In order to prove $(c)$ implies $(d)$ we define $$R=\{(n,x,y)\in I\times X\times X\mid y\in N(n,x)\}$$ and put $p(n,x,y)=k$ for some $k\in I$ such that $N(k,y)\subseteq N(n,x)$, which exists by assumption on condition $(c)(ii)$. Once again, conditions $(d)(i)$ and $(d)(iii)$ are direct consequences of $(c)(i)$ and $(c)(iii)$, respectively. To see that $(d)(ii)$ is satisfied we observe that if $y\in N(n,x)$ and $z\in N(p(n,x,y),y)$, then, by definition of $p(n,x,y)$, we have $N(k,y)\subseteq N(n,x)$. It follows $z\in N(n,x)$ and hence $(d)(ii)$ is satisfied. Having $R$ we define $N_R=\{y\mid (n,x,y)\in R\}=\{y\in N(n,x)\}=N(n,x)$ and so the topology $\tau$ is obtained by $N_R=N$.

To prove $(d)$ implies $(e)$, define $\eta(i,x)=\{y\in X\mid (n,x,y)\in R\}=N_R(n,x)$ and put \[\gamma((U,x))=k\] for some $k\in I$ such that $N_R(k,x)\subseteq U$, which exists by the assumption that $\tau$ is generated by $N_R$. The map $\eta\colon{I\times X\to\tau}$ is well defined because each $\eta(n,x)\in \tau$. Indeed, if $y\in \eta(n,x)$, that is $(n,x,y)\in R$, then there exists $k=p(n,x,y)\in I$ for which $N_R(k,y)\subseteq N_R(n,x)$. This is a consequence of $(d)(ii)$. If $z\in N_R(k,y)\Leftrightarrow (k,y,z)\in R$ then, given that $(n,x,y)\in R$ and $k=p(n,x,y)$, we have $(n,x,y)\in R$ or, in other words, $z\in N_R(n,x)$. This shows that each $\eta(n,x)$ is open in $\tau$. Conditions $(e)(i)$ and $(e)(iii)$ are direct consequence of $(d)(i)$ and $(d)(iii)$, respectively. To show $(e)(ii)$ we observe that if $y\in\eta(\gamma(U,x),x)$ then $(\gamma(U,x),x,y)\in R$, but $\gamma(U,x)=k$, for some $k\in I$ such that $N_R(k,x)\subseteq U$. Since $y\in N_R(k,x)\Leftrightarrow(k,x,y)\in R$, we conclude that $y\in U$. This shows that $\eta(\gamma(U,x),x)\subseteq U$.

Finally, we prove $(e)$ implies $(b)$. Having $\eta$ and $\gamma$ it is not difficult to see that a cartesian spacial fibrous preorder is obtained if we let
\begin{eqnarray*}
x\leq^{i}y \Leftrightarrow y\in \eta(i,x)\\
\partial^{i}(x,y)=\gamma(\eta(i,x),y)
\end{eqnarray*}
with $i\in I$ and $x,y\in X$. Indeed, axioms $(C1)$ and $(C3)$ follow respectively from $(e)(i)$ and $(e)(iii)$. For $(C2)$ let us suppose $x\leq^i y$, that is, $y\in\eta(i,x)$, and let us suppose $y\leq^k z$ with $k=\partial^i(x,y)=\gamma(\eta(i,x,y))$. This means $z\in\eta(k,y)$ and by condition $(e)(ii)$ we have $\eta(\gamma(\eta(i,x),y),y)\subseteq \eta(n,x)$, thus we have $z\in \eta(i,x)$, or $x\leq ^i z$ as desired.

It remains to show that the topology $\tau$ is recovered as prescribed in (\ref{eq: tau (copy)}). On the one hand if $\mathcal{O}\in \tau$ and if $x\in \mathcal{O}$ then there is $k=\gamma(\mathcal{O},x)$ with $N(k,x)\subseteq \mathcal{O}$. This means that every open set in $\tau$ is generated as in condition (\ref{eq: tau (copy)}). To see the converse let us consider any subset $\mathcal{O}\subseteq X$ and suppose it has the property that for all $x\in \mathcal{O}$ there is some $k=k(x)\in I$ with $N(k,x)=\eta(k,x)\subseteq \mathcal{O}$. We have to show that $\mathcal{O}\in \tau$. This follows because \[\mathcal{O}=\bigcup_{x\in \mathcal{O}}N(k(x),x)\] and  every $N(k,x)=\eta(k,x)\in \tau$.
\end{proof}

%The previous result suggest to call \emph{$I$-cartesian} to any topological space $(X,\tau)$ satisfying any one of the equivalent conditions stated above.

We immediately observe some interesting special cases, namely when $I=\{1\}$ is the trivial unitary magma, or when $I=(\mathbb{N},\cdot,1)$ is the unitary magma (monoid) of natural numbers with the usual multiplication. Furthermore, as we will see in Section \ref{monoids and their topologies}, the map $p\colon{R\to I}$ of condition $(d)$ in the previous result can sometimes be decomposed as $p(n,x,y)=\beta(x,y)\cdot n$ in which $\beta(x,y)=\gamma(\eta(1,x),y)$ with $\eta$ and $\gamma$ as in condition $(e)$ above.

In a sequel to this work our attention will be turned to morphisms between fibrous preorders and on how they can be defined internally to any category with finite limits. 
%Then, we will be able to formulate the previous result as follows: \emph{For every fixed unitary magma, $(I,\cdot,1)$, there is an equivalence of categories between the category of cartesian spacial fibrous preorders indexed by $I$ and the category of $I$-cartesian spaces.} 

 For the moment, let us briefly mention that a morphism between cartesian spacial fibrous preorders, say from $$(X,(\leq^{i})_{i\in I},(\partial^{i})_{i\in I})$$ to $$(Y,(\leq^{i})_{i\in I},(\partial^{i})_{i\in I}),$$ consists of a map $f\colon{X\to Y}$ and a family of maps $(g_j\colon{X\to I})_{j\in I}$ such that
\begin{equation}\label{eq: fibrous morphism}
x\leq^{g_j(x)} y \Rightarrow f(x)\leq^{j} f(y)
\end{equation}
for all $x,y\in X$ and $j\in I$. Note that when each $g_j$ can be chosen independently from $x$ then the family $(g_j)_{j\in I}$ may be seen as a single map $g\colon{I\to I}$ and this is comparable to the notion of uniform continuity (see the example of metric spaces presented above). It is also clear that morphisms between cartesian spacial fibrous preorders may be considered for different indexing unitary magmas. This and other considerations will be explored in a sequel to this work.

%Let us from here on restrict our attention to those topological spaces which are the realization of a cartesian spacial fibrous preorder.
%
%\begin{proposition}\label{thm: I-spaces (natural)}
%Let $(I,\cdot,1)$ be a unitary magma and $X$ be any set. Every map $N\colon{I\times X\to \mathcal{P}(X)}$ into the power-set of $X$ and satisfying the following three conditions:
%\begin{eqnarray}
%x\in N(i,x)\label{eq: I-spaces 1}\\
%N(i,x)\subseteq \{y\in X\mid \exists j\in I, N(j,y)\subseteq N(i,x)\}\label{eq: I-spaces 2}\\
%N(i\cdot j,x)\subseteq N(i,x)\cap N(j,x)\label{eq: I-spaces 3}
%\end{eqnarray}
%for all $x\in X$, and $i,j\in I$, gives rise to a cartesian spacial fibrous preorder indexed by $I$ and hence to a topology on $X$.
%\end{proposition}
%\begin{proof}
%Define $x\leq^{i} y$ whenever $y\in N(i,x)$ and choose $\partial^{i}(x,y)=j\in I$ as any element in $I$ such that $N(j,y)\subseteq N(i,x)$ whose existence is ensured by condition $(\ref{eq: I-spaces 2})$.
%\end{proof}

 So far we have isolated one basic ingredient in the structure of those topological spaces which are obtained as cartesian spacial fibrous preorders: a unitary magma. Note that every ascending chain on a set $I$ makes it a unitary magma with the smallest element as neutral and as binary operation the procedure of choosing the greatest between any two. 
 
 We now ask: Besides a unitary magma, is it possible that the structure of a system $(X,(\leq^{i})_{i\in I},(\partial^{i})_{i\in I})$, inducing a topology on $X$, can further be decomposed into simpler structures which are hence easier to analyse?

The answer to this question, as we will see, is yes provided only that we restrict its generality a little bit further. Nevertheless, we are still able to capture the majority of examples of which we are interested in, namely preorders, metric spaces and certain special classes of groups and monoids with a topology (see Sections \ref{the structure of metric spaces} and \ref{monoids and their topologies}). First, let us analyse the concrete examples.

\section{Some examples of cartesian spacial fibrous preorders}

Let us start with two simple cases of interest as unitary magmas. The trivial (singleton) unitary magma, and the unitary magma of natural numbers with usual multiplication as its binary operation.

 It is clear that a cartesian spacial fibrous preorder, indexed by a singleton unitary magma $I=\{1\}$, is nothing but a preorder. Indeed, in that case we always have $\partial^{i}(x,y)=i$ with $i=1$ and hence condition (C2) becomes transitivity. Condition (C1) asserts reflexivity while condition (C3) is trivial.
 
 \begin{proposition}\label{thm: preorders and 1-spaces}
The category of preorders is equivalent to the category of $1$-cartesian topological spaces which are the same as Aleksandrov or discrete spaces. 
 \end{proposition}
\begin{proof}
Preorders are precisely cartesian spacial fibrous preorders indexed by a singleton unitary magma. Moreover, if $I=\{1\}$ then condition $(\ref{eq: fibrous morphism})$ asserts the monotonicity of the map $f$. The result follows from Theorem \ref{thm: I-cartesian spaces}. It is well-known that Aleksandrov spaces, or discrete spaces, are the same as preorders \cite{Aleksandrov 1937} (see also \cite{NMF 2014}).
\end{proof}

When $I=(\mathbb{N},\cdot,1)$ is the unitary magma of natural numbers with usual multiplication then we have several interesting cases.

A natural space \cite{NMF 2014} is essentially a first-countable topological space. It consists of a pair $(X,N)$ with $X$ a set and $N\colon{\mathbb{N}\times X\to \mathcal{P}(X)}$ a map into the power-set of $X$ satisfying the following three conditions for all $x\in X$, and $i,j\in \mathbb{N}$:
\begin{eqnarray}
x\in N(i,x)\label{eq:natural spaces 1}\\
N(i,x)\subseteq \{y\in X\mid \exists j\in \mathbb{N}, N(j,y)\subseteq N(i,x)\}\label{eq:natural spaces 2}\\
N(ij,x)\subseteq N(i,x)\cap N(j,x).\label{eq:natural spaces 3}
\end{eqnarray}

Natural spaces are precisely $\mathbb{N}$-cartesian spaces (see item $(c)$ of Theorem $\ref{thm: I-cartesian spaces}$).

In a similar manner, every metric space $(X,d)$ gives rise to a cartesian spatial fibrous preorder indexed by the natural numbers as already presented. Put $x\leq^{n} y$ whenever $d(x,y)< \frac{1}{n}$ and define $\partial^{n}(x,y)=k$ such that $\frac{1}{k}\leq \frac{1}{n}-d(x,y)$.

Similarly, every normed vector space gives rise to a cartesian spatial fibrous preorder indexed by the natural numbers. We simply put $x\leq^{n} y$ whenever $\|y-x\|< \frac{1}{n}$ and define $\partial^{n}(x,y)=k\in \mathbb{N}$ as any natural number such that $\frac{1}{k}\leq \frac{1}{n}-\|y-x\|$. However, if reformulated as:
\begin{eqnarray}
x\leq^{n} y\Leftrightarrow n\|y-x\|<1
\end{eqnarray}
and $\partial^{n}(x,y)=rn$ with $r\in \mathbb{N}$ any natural number such that for all $u\in X$
\begin{eqnarray}
r\|u\|<1\Rightarrow \|u+n(y-x)\|<1,
\end{eqnarray}
then we observe that the needed topological information is reduced to:
\begin{enumerate}
\item the open ball, $B_1$, of radius 1 and centred at the origin;
\item  a choice of a natural number $r=r(z)$, for every $z\in B_1$, such that for all $u\in X$
\begin{eqnarray}
r\|u\|<1\Rightarrow \|u+z\|<1.
\end{eqnarray}
\end{enumerate}

This is sufficient motivation to consider those topological groups $(X,0,+,\tau)$ for which there exists an open neighbourhood of the origin, say $B\subseteq X$, satisfying the following  condition:
\begin{eqnarray}
\forall x\in B, \exists n\in \mathbb{N}, B_n+x\subseteq B,
\end{eqnarray} 
which, in other words, if defining $B_n=\{u\in X\mid nu\in B\}$, can be stated as
\begin{eqnarray}
\forall x\in B, \exists n\in \mathbb{N},\forall u\in X, nu\in B\Rightarrow u+x\in B.
\end{eqnarray} 

We observe, furthermore, that there is no need to start with a topological group. Any group with a subset $B$, containing the origin and satisfying the previous condition, immediately satisfies conditions (C1) and (C2) in the definition of a cartesian spacial fibrous preorder as soon as we put 
$x\leq^{i} y$ whenever $n(y-x)\in B$ and define 
 $\partial^{n}(x,y)=rn$ with $r\in \mathbb{N}$ any natural number such that 
\begin{eqnarray}\label{eq: group with open neighbourhood of the origin}
\forall u\in X, ru\in B\Rightarrow u+n(y-x)\in B.
\end{eqnarray} 

Indeed, for every $n\in \mathbb{N}$, we get $x\leq^{n} x$ if and only if $n(x-x)\in B$ or equivalently $0\in B$. Now, if $x\leq^{n} y$ and $y\leq^{rn} z$ with $r\in \mathbb{N}$ such that $(\ref{eq: group with open neighbourhood of the origin})$ holds, then as a consequence we have $x\leq^{n} z$. To see it let us take any $n(y-x)\in B$ and $rn(z-y)\in B$, then, by taking $u=n(z-y)\in X$, we observe $ru\in B$ and hence $u+n(y-x)\in B$ or equivalently $n(z-y)+n(y-x)\in B$ which is the same as $n(z-x)\in B$, precisely stating that $x\leq^{n} z$.
 
In order to prove condition (C3) we need an extra assumption on the subset $B$, namely that $B_n\subseteq B$ for every $n\in \mathbb{N}$. This is clearly the case for normed vector spaces and the previous considerations can be summarized as follows.

\begin{proposition}\label{thm: a group with a subset gives a fibrous etc}
Let $(X,+,0)$ be a group and suppose there are $B\subseteq X$ and $\alpha\colon{B\to\mathbb{N}}$ satisfying for every $x\in B$, $u\in X$ and $n\in \mathbb{N}$ the following three conditions:
\begin{enumerate}
\item[(i)] $0\in B$;
\item[(ii)]  if $\alpha(x)u\in B$ then  $u+x\in B$
\item[(iii)] if $nu\in B$ then $u\in B$
\end{enumerate} 
Then, the system $(X,\leq^{n},\partial^{n})$, defined as $x\leq^{n} y$ whenever $n(y-x)\in B$ and 
 $\partial^{n}(x,y)=\alpha(n(y-x))n$, is a cartesian spacial fibrous preorder indexed by the natural numbers. Thus giving rise to an $\mathbb{N}$-cartesian space on the set $X$.
\end{proposition} 

In the following section we will analyse in more detail the example of a metric space. The example of a normed vector space, which has been generalized by the previous proposition, will be analysed further in Section \ref{monoids and their topologies}.
 
\section{The structure of metric spaces}\label{the structure of metric spaces}

In this section we will see how to decompose the structure of a metric space in terms of a cartesian spacial fibrous preorder. First, instead of the non-negative real interval $[0,+\infty[$, we can take any preorder $(E,\leq)$. Secondly, instead of the formula $\frac{1}{n}-d(x,y)$ we will use maps $p(\alpha,x,y)=\alpha-d(x,y)$ and $g(n)=\frac{1}{n}$ so that $\frac{1}{n}-d(x,y)$ is obtained by taking $p(g(n),x,y)$. These maps take their values on the preorder $(E,\leq)$ and are required to satisfy some conditions. The definition of $x\leq^n y$ is recovered as $g(k)\leq p(g(n),x,y)$ for some $k\in I$, which is then used to define $\partial^n(x,y)=k$.

Let $(E,\leq)$ be a preorder and $B$ a set. We will say that a map $p\colon{E\times B\times B\to E}$ is a \emph{lax-left-associative Mal'tsev operation} when the following three conditions are satisfied for all $a,b\in E$ and $x,y,z\in B$:
\begin{enumerate}
\item $a\leq p(a,x,x)$,
\item $p(p(a,x,y),y,z)\leq p(a,x,z)$,
\item if $a\leq b$ then $p(a,x,y)\leq p(b,x,y)$.
\end{enumerate}
 
Let $(I,\cdot,1)$ be a unitary magma and $(E,\leq)$ a preorder. We will say that a map $g\colon{I\to E}$ is a \emph{linking map} if \[g(n\cdot(k\cdot m))\leq g(n\cdot m)\] for all $n,m,k\in I$.
 
\begin{proposition}\label{thm: metric spaces} Let $(I,\cdot,1)$ be a unitary magma and $(E,\leq)$ a preorder. Every lax-left-associative Mal'tsev operation $p\colon{E\times B\times B\to E}$ together with a linking map $g\colon{I\to E}$ induces a topology $\tau$ on the set $B$ determined by
\[\mathcal{O}\in \tau \Leftrightarrow \forall x\in \mathcal{O}, \exists n\in I, N(n,x)\subseteq \mathcal{O}
\]
with $N(n,x)=\{y\in B\mid \exists m\in I, g(m)\leq p(g(n),x,y)\}$.
\end{proposition}
\begin{proof}
The proof makes use of Proposition \ref{thm: fibrous to spaces} by showing that the system $(B,\leq^{n},\partial^{n})$ is a cartesian spacial fibrous preorder with \[x\leq^{n}y \Leftrightarrow  \exists m\in I, g(m)\leq p(g(n),x,y)\] 
and $\partial^n(x,y)=m$.

Axiom (C1) holds because $g(n)\leq p(g(n),x,x)$ for all $n\in I$ and $x\in B$.

In order to prove Axiom (C2) we consider $x\leq^n y$ with $\partial^n(x,y)=m$ such that $g(m)\leq p(g(n),x,y)$ and $y\leq^m z$ for which there exists $m'=\partial^m (y,z)\in I$ such that $$g(m')\leq p(g(m),y,z).$$
Under these assumptions, having in mind that $p$ is a lax-left-associative Mal'tsev operation, we observe
\[g(m')\leq p(g(m),y,z)\leq p(p(g(n),x,y),y,z)\leq p(g(n),x,z)\]
which shows that $x\leq^n z$.

Axiom (C3) is a consequence of $g$ being a linking map from which, in particular, we obtain
\[g(k\cdot m)\leq g(m)\]
and
\[g(n\cdot k)\leq g(n).\]
In the former case $n=1$ while in the latter $m=1$. This allows us to conclude that if $x\leq^{n\cdot k}y$ then $x\leq^n y$. Indeed, if there exists $m\in I$ such that $g(m)\leq p(g(n\cdot k),x,y)$ then
\[p(g(n\cdot k),x,y)\leq p(g(n),x,y)\]
and so $g(m)\leq p(g(n),x,y)$, thus ensuring $x\leq^n y$. Similarly we prove that if $x\leq^{k\cdot m}y$ then $x\leq^m y$.
\end{proof} 
 
We remark that the reflexivity of the preorder $(E,\leq)$ is never used, so that the previous result still holds for a set $E$ equipped with a transitive relation. So, for example, the above result holds true if we replace the preordered set $(E,\leq)$ by a semigroup $(E,\circ)$ and define the transitive relation $x<y$ as $x\circ y=x$. As it is well known this relation fails to be reflexive when $E$ is not an idempotent semigroup and fails to be anti-symmetric when $E$ is not commutative. In addition, the requirement on the map $g$ being a linking map could be replaced by the condition that $g(n\cdot m)\leq g(n)$ {and} $g(n\cdot m)\leq g(m)$.

 The example of a metric space is obtained by letting $E$ be the set of real numbers with the usual order, and taking the maps $p(a,x,y)=a-d(x,y)$ and $g(n)=\frac{1}{n}$ with $I$ the unitary magma of natural numbers with usual multiplication.

 Another interesting example is obtained by taking  $E$ to be an ordered group $(E,+,0,\leq)$ and $I=(I,\cdot,1)$ a monoid. Suppose there exists a map $g\colon{I\to E}$ with $g(mkn)\leq g(mn)$ for all $m,n,k\in I$. Let $B$ be any set and consider a map $$\delta\colon{B\times B\to E}$$ such that 
 \begin{eqnarray}
 \delta(x,x)=0, \quad \text{for all $x\in B$}\\
 \delta(x,y)+\delta(y,z)\geq \delta(x,z),\quad \text{for all $x,y,z\in B$.}
 \end{eqnarray}
 This is a straightforward generalization of a metric space and we have that $p(a,x,y)=a-\delta(x,y)$ is a lax-left-associative Mal'tsev operation.
 
 Let us now suppose that $B$ is a group. Then, with $g$ and $E$ as before, for every map $t\colon{B\to E}$ such that $ t(0)=0$ and
\begin{eqnarray}
 t(u)+t(v)\geq t(u+v),\quad \text{for all $u,v\in B$}
 \end{eqnarray}
 we get a lax-left-associative Mal'tsev operation with $p(a,x,y)=a-t(y-x)$. This is an immediate generalization for normed spaces. In particular, when $t$ is a homomorphism, we get $p(a,x,y)=a+t(x)-t(y)$.
  
A simple procedure to construct unitary magmas which in general are not associative is to start with an arbitrary non-empty set of indexes $I$, choose and element $1\in I$ and consider a family of endo-maps $\mu_n\colon{I\to I}$, indexed by the elements in $I$, such that $\mu_n(1)=1$ and $\mu_1(n)=n$ for all $n\in I$. A binary operation is thus obtained as $m\cdot n=\mu_m(n)$.

One final remark on notation. The notion of a linking map comes from the structure of a link which is part of work in progress. The name lax-left-associative Mal'tsev operation is due to the fact that when $E=B$ and the preorder is the identity or discrete order, then, a lax-left-associative Mal'tsev operation reduces to   
\begin{eqnarray*}
a= p(a,x,x)\\
p(p(a,x,y),y,z)= p(a,x,z).
\end{eqnarray*}
If adding the respective two similar identities on the right $ p(x,x,a)=a$ and $p(z,x,a)=p(z,y,p(y,x,a)$, which make sense because $E=B$, then we would obtain an associative Mal'tsev operation.

\section{Monoids and modules as cartesian spaces}\label{monoids and their topologies}

An interesting special class of $I$-cartesian spaces is obtained by imposing on the structuring maps $\gamma$ and $\eta$ of Theorem \ref{thm: I-cartesian spaces}(e) the condition
\[\gamma(\eta(n,x),y)=\gamma(\eta(1,x),y)\cdot n\]
for all $y\in \eta(n,x)$ and for all $n\in I$ and $x\in X$.

We will analyse this condition by considering a monoid structure on the set $B$. This will allow us to decompose $\eta(n,x)=x+\eta(n,0)$. We then consider an $I$-module structure on $B$ thus providing a way to obtain $\eta(n,0)$ as an $n$-scaling of $\eta(1,0)$. Therefore, when $B$ is an $I$-module, all the information is encompassed in $\eta(1,0)$ and $\gamma(\eta(1,0),z)$.

\begin{proposition}
Let $(I,\cdot,1)$ be a unitary magma. Every monoid $(B,+,0)$, equipped with a family of subsets $S_n\subseteq B$ together with maps
\[\alpha_n\colon{S_n\to I,\quad n\in I}\]
such that
\begin{enumerate}
\item[(i)] $0\in S_n$, for all $n\in I$
\item[(ii)] for each $n\in I$, if $a,a'\in S_n$ and $a'\in S_{\alpha_n(a)}$ then $a+a'\in S_n$
\item[(iii)] $S_{n\cdot m}\subseteq S_n\cap S_m$
\end{enumerate}
induces a topology $\tau$ on $B$ determined as
\[\mathcal{O}\in \tau \Leftrightarrow \forall x\in \mathcal{O}, \exists n\in I, x+S_n\subseteq \mathcal{O}.
\]
\end{proposition} 
\begin{proof}
We make use of Proposition \ref{thm: fibrous to spaces} by showing that the system $(B,(\leq^n )_{n\in I},(\partial^n)_{n\in I})$ is a cartesian  spacial fibrous preorder with
\[x\leq^n y\Leftrightarrow \exists a\in S_n, x+a=y\] 
and $\partial^n(x,y)=\alpha_n(a)\cdot n$ while noting that $N(n,x)=\{y\in B\mid x\leq^n y\}$ is the same as $x+S_n=\{x+a\mid a\in S_n\}$.

Axiom (C1) holds because $0\in S_n$ and hence $x\leq^n x$ for all $x\in B$ and $n\in I$.

In proving axiom (C2) we observe that if $x\leq^n y$, that is, $x+a=y$ for some $a\in S_n$, and if $y\leq^m z$, with $m=\alpha_n(a)\cdot n$, which means $y+a'=z$ for some $a'\in S_m$, then by condition $(iii)$ we conclude that $a'\in S_n$ and $a'\in S_{\alpha_n(a)}$. Condition $(ii)$ now tells us that $a+a'\in S_n$ and hence $x\leq^n z$. Indeed, there exists $a''=a+a'\in S_n$ such that $x+a''=z$.

Axiom (C3) is a straightforward consequence of condition $(iii)$. If $x\leq^{n\cdot m}y$, that is $x+a=y$ for some $a\in S_{n\cdot m}$, then $x\leq^n y$ and $x\leq^m y$ since, by $(iii)$, $a\in S_n$ and $a\in S_m$.
\end{proof}

In particular, if there exists an action of $I$ on $B$, in the sense of a map $\xi\colon{I\times B\to B}$ such that   for all $n,m\in I$ and $x,y\in B$
\begin{enumerate}
\item  $\xi(1,x)=x$, $\xi(n,0)=0$,
\item $\xi(n,x+y)=\xi(n,x)+\xi(n,y)$
\item $\xi(n\cdot m,x)=\xi(n,\xi(m,x))=\xi(m,\xi(n,x))$,
 \end{enumerate}
  then we may wonder if each $S_n$ in the previous proposition is determined from $S_1$ as $S_n=\{\xi(n,a)\mid a\in S_1\}$.

A monoid $(B,+,0)$ equipped with an action $\xi$ of $I$ on $B$ in the sense above is said to be an $I$-module and represented as $(B,+,0,\xi)$.
 
\begin{theorem}\label{thm: I-module}
Let $(I,\cdot,1)$ be a unitary magma, $(B,+,0,\xi)$ an $I$-module and $(B,\tau)$ a topological space. The following are equivalent:
\begin{enumerate}
\item[(a)] There exists $S\subseteq B$ and $\alpha\colon{S\to I}$ such that 
\begin{enumerate}
\item[(i)] $0\in S$
\item[(ii)] $a+\xi(\alpha(a),a')\in S$, for all $a,a'\in S$
\end{enumerate}
for which $\tau$ is determined as
\[\mathcal{O}\in \tau \Leftrightarrow \forall x\in \mathcal{O}, \exists n\in I, x+S_n\subseteq \mathcal{O}
\]
with $S_n=\{\xi(n,a)\in B\mid a\in S\}$.
\item[(b)] There exists $S\subseteq B$ such that 
\begin{enumerate}
\item[(i)] $0\in S$
\item[(ii)] $S\subseteq \{y\in B\mid \exists n\in I, y+S_n\subseteq S\}$, with $$S_n=\{\xi(n,a)\in B\mid a\in S\}$$
\end{enumerate}
for which $\tau$ is determined as
\[\mathcal{O}\in \tau \Leftrightarrow \forall x\in \mathcal{O}, \exists n\in I, x+S_n\subseteq \mathcal{O}.
\]
\item[(c)] There exists $S\in\tau$ and $\gamma\colon{\{(U,x)\mid x\in U\in \tau\}\to I}$ such that
\begin{enumerate}
\item[(i)] $0\in S$
\item[(ii)] $x+S_n \in \tau$ for all $x\in B$ and $n\in I$ with $$S_n=\{\xi(n,a)\in B\mid a\in S\}$$
\item[(iii)] if $x\in U\in\tau$ and $n=\gamma(U,x)$ then $x+S_n\subseteq U$. 
\end{enumerate}
\end{enumerate}
Moreover, under these conditions, $(B,\tau)$ is an $I$-cartesian space with $x\leq^n y$ defined as $y=x+\xi(n,a)$ for some $a\in S$ and $\partial^n(x,y)=\alpha(a)\cdot n$ if and only if $\xi(n,a)\in S$ for all $a\in S$ and $n\in I$.
\end{theorem} 
\begin{proof}
$(a)\Rightarrow (b)$ We only need to prove that $S$ is contained in $$\{y\in B\mid \exists n\in I, y+S_n\subseteq S\}$$ with $S_n=\{\xi(n,a)\in B\mid a\in S\}$. This is the same as proving that for every $a\in S$ there exists $n\in I$ for which $a+S_n\subseteq S$. It follows from condition $(a)(ii)$ by choosing $n=\alpha(a)$. Indeed, for every $a\in S$, we have $a+S_{\alpha(a)}\subseteq S$ as a consequence of $S_{\alpha(a)}=\{\xi(\alpha(a),a')\mid a'\in S\}$ together with condition $(a)(ii)$.

$(b)\Rightarrow (c)$ Let $x\in B$ and $n\in I$. In order to show that $x+S_n\in \tau$ we consider $z=x+\xi(n,a)$ for some $a\in S$ and find $m\in I$ for which $z+S_m\subseteq x+S_n$. Using condition $(b)(ii)$ and given that $a\in S$, we obtain $k\in I$ such that $a+S_k\subseteq S$, or, in other words, $a+\xi(k,a')\in S$ for all $a'\in S$. We now take $m=k\cdot n$ and show $z+S_m\subseteq x+S_n$.Indeed, for every $a'\in S$, taking into account that $\xi$ is an action, we observe
\begin{eqnarray*}
z+\xi(m,a') &=& z+\xi(k\cdot n,a')=z+\xi(n,\xi(k,a'))\\
&=& x+\xi(n,a)+\xi(n,\xi(k,a'))\\
&=& x+\xi(n,a+\xi(k,a'))
\end{eqnarray*}
showing that $z+\xi(m,a')=x+\xi(n,a+\xi(k,a'))$ is in $x+S_n$ because $a+\xi(k,a')\in S$ for all $a'\in S$. In particular, when $x=0$ and $n=1$ we obtain $S\in \tau$. Thus, so far we have shown the existence of $S$ satisfying condition $(c)(ii)$. Condition $(c)(i)$ is clear while Condition $(c)(iii)$ is obtained by observing that if $x\in U\in \tau$ then, by the assumption on $\tau$, there exists $n\in I$ for which $x+S_n\subseteq U$, and  hence we choose $\gamma(U,x)=n$ in order to satisfy condition $(c)(iii)$.

$(c)\Rightarrow (a)$ Given $S\in \tau$, the map $\alpha\colon{S\to I}$ is defined as $\alpha(a)=\gamma(S,a)$. From condition $(c)(iii)$ it follows that \[a+S_{\alpha(a)}\subseteq S\]
which means that for every $a,a'\in S$,
\[a+\xi(\alpha(a),a')\in S.\]
This proves condition $(a)(ii)$. Condition $(a)(i)$ is clear. It remains to prove that $\tau$ is determined as prescribed. We observe, on the one hand, if $\mathcal{O}\in \tau$ then for every $x\in \mathcal{O}$ there exists $m=\gamma(\mathcal{O},x)\in I$ such that $x+S_m\subseteq \mathcal{O}$ as asserted by condition $(c)(iii)$. On the other hand, if $\mathcal{O}\subseteq B$ is such that for all $x\in \mathcal{O}$ there exists $m\in I$ with $x+S_m\subseteq \mathcal{O}$, then, because $x+S_m\in\tau$ and \[\mathcal{O}=\bigcup_{x\in \mathcal{O}}x+S_m\]
we conclude that $\mathcal{O}\in \tau$.

Moreover, under these conditions $(B,\tau)$ is $I$-cartesian, where the map $\eta$ (see item $(e)$ in Theorem \ref{thm: I-cartesian spaces}) is obtained as $\eta(n,x)=x+S_n$, if, and only if $S_n\subset S$ for all $n\in I$. Indeed, if $S_n\subset S$ for all $n\in I$ we observe:
\begin{enumerate}
\item[(i)] $x\in \eta(n,x)$, since $0\in S$ and $\xi(n,0)=0$.
\item[(ii)] if $x\in U\in\tau$ and $m=\gamma(U,x)$ then $x+S_m\subseteq U$, which is the same as $\eta(\gamma(U,x),x)\subseteq U$.
\item[(iii)] $\eta(n\cdot m,x)\subseteq \eta(n,x)\cap \eta(m,x)$ follows from $S_{n\cdot m}\subseteq S_n\cap S_m$ which is a consequence of $\xi(n,a)\in S$ for all $n\in I$ and $a\in S$ and the fact that $\xi$, being an action, is such that $\xi(n\cdot m,a)=\xi(n,\xi(m,a))=\xi(m,\xi(n,a))$.
\end{enumerate}

Conversely, if $(B,\tau)$ is $I$-cartesian with $\eta(n,x)=x+S_n$ then from $\eta(n\cdot m,x)\subseteq \eta(n,x)\cap \eta(m,x)$ for all $m,n\in I$ we deduce $S_n\subseteq S$ when $m=1$.
\end{proof}

%In the following section we will analyse some concrete examples. 

%When $B$ is a group, as shown in Proposition \ref{thm: I-groups}, it is possible to consider pseudo-actions instead of actions in the sense that the condition $\xi(n,x+y)=\xi(n,x)+\xi(n,y)$ may not be verified. 
 
\section{Examples} 

Some examples are presented so to illustrate the results on the previous section.

Theorem \ref{thm: I-module} can be specialized into the case when the unitary magma is the set of natural numbers with usual multiplication $I=(\mathbb{N},\cdot,1)$, and $(B,+,0)$ is any monoid considered as an $I$-module with $\xi(n,x)=nx$. In this case, any subset $S\subseteq B$ together with a map $\alpha\colon{S\to \mathbb{N}}$ satisfying the three conditions:
\begin{enumerate}
\item[(i)] $0\in S$
\item[(ii)] $a+\alpha(a)a'\in S$ for all $a,a'\in S$
\item[(iii)] $na\in S$ for all $a\in S$ and $n\in\mathbb{N}$.
\end{enumerate}
induces a topology $\tau$ on $B$ generated by the system of open neighbourhoods $N(n,x)=\{x+na\mid a\in S\}$.

Note that condition $(iii)$ is necessary so that the system $(B,\leq^n,\partial^n)$ with  $x\leq^n y$ whenever $y=x+na$ for some $a\in S$ and $\partial^n(x,y)=\alpha(a)n$ to be a cartesian spacial fibrous preorder. When $B$ is a group it offers a further comparison with Proposition \ref{thm: a group with a subset gives a fibrous etc} which will be deepened in a future work. 

When the map $\alpha\colon{S\to \mathbb{N}}$ is constant with $\alpha(a)=1$ then $S$ must be a submonoid.

Another example is obtained when $B=(B,+,\cdot,0,1)$ is a semi-ring. In this case we can take $I=(\mathbb{N}_0,+,0)$ and choose any submonoid $S$ for the additive structure, that is, $S\subseteq B$ with $0\in S$ and $a+a'\in S$ for all $a,a'\in S$. Now, every choice of an element $p\in S$ such that $p^n a\in S$, for all $a\in S$ and $n\in \mathbb{N}_0$, gives a topology on the set $B$ generated by the system of open neighbourhoods
\[N(n,x)=\{x+p^n a\mid a\in S\}.\]

Indeed, we define an $I$-module with action $\xi(n,x)=p^nx$, by taking successive powers of $p$ and considering $p^0=1$. In this case the map $\alpha\colon{S\to\mathbb{N}_0}$ is the constant map $\alpha(a)=0$.

When $S$ is not closed under addition then we have to find for every $a\in S$ a non-negative integer $\alpha(a)\in \mathbb{N}_0$ such that $a+p^{\alpha(a)}a'\in S$ for all $a'\in S$. The simple example of the real numbers $(\mathbb{R},+,\cdot,0,1)$ with usual addition and multiplication, when $S=[0,1[$ and $p=\frac{1}{2}$, illustrates this situation.

One more example with $I=(\mathbb{N}_0,+,0)$ is constructed as follows. 

Take any monoid $(X,+,0)$ together with $t\colon{X\to X}$, an endomorphism of $X$, and consider a subset $P\subseteq X$ with $0\in P$. Let $B$ be the collection of all maps from $X$ to $X$ considered as a monoid with component-wise addition, that is, $B=(X^X,+,0_X)$. With $I=(\mathbb{N}_0,+,0)$ we define an action on $B$ as $\xi(n,f)=t^n\circ f$ for every non-negative integer $n$ and every map $f\colon{X\to X}$. The action is defined for every selected endomorphism $t$ of $X$.

In order to produce a topology on the set $B$ we need to find a subset $S\subseteq B $ together with a map $\alpha\colon{S\to \mathbb{N}_0}$ as in Theorem \ref{thm: I-module}. We take
\[S=\{f\in B\mid \exists n\in \mathbb{N}_0,\forall x\in X,\forall y\in P,f(x)+t^n(y)\in P\}\]
and put $\alpha(f)=n$ where $n$ is the smallest non-negative integer such that for all $x\in X$ and $y\in P$, $f(x)+t^{n}(y)\in P$.

Clearly, the constant null function $0_X\colon{X\to X}$ is in $S$ and moreover $\alpha(0_X)=0$.

The reason why $f+t^{\alpha(f)}\circ g\in S$ for all $f,g\in S$ is mainly because $t$ is an endomorphism and so $t^n\circ g+t^{n+m}=t^n\circ g+t^n\circ t^m=t^n\circ(g+t^m)$. Indeed, if $f,g\in S$ with $\alpha(f)=n$ and $\alpha(g)=m$ then there exists $k=n+m$ such that $(f+t^n\circ g)(x)+t^k(y)\in P$ for all $x\in X$ and $y\in P$.

The imposition that $t^n\circ f\in S$ for all $n\in \mathbb{N}_0$ and $f\in S$ is guaranteed as soon as the endomorphism $t\colon{X\to X}$ satisfies the condition $t(y)\in P$ for all $y\in P$. In that case $\alpha(t^n\circ f)=n+\alpha(f)$.

A concrete simple example is the following. Take $X=(]0,1],\cdot,1)$ to be the unit interval of  real numbers with usual multiplication and not containing the element $0$. Consider the endomorphism $t(x)=\sqrt{x}=x^{\frac{1}{2}}$ and let $P=]\frac{1}{2},1]$. It follows that $f\in S$ if and only if $f(x)\in P$ for all $x\in X$. For every $f\in S$, $\alpha(f)=n$ with $n$ being the smallest non-negative integer for which
\[u\cdot\left(\frac{1}{2}\right)^{\frac{1}{2^n}}>\frac{1}{2}\]
with $u=\inf{\{f(x)\mid x\in X\}}$. Note that $u$ is always greater than $\frac{1}{2}$, and 
\[\lim_{n\to+\infty}\left(\frac{1}{2}\right)^{\frac{1}{2^n}}=1.\]

As a consequence of Theorem \ref{thm: I-module} we obtain a topology on $B$ generated from
\[N(n,u)=\{u+t^n\circ v\mid v\in S\}\]
as system of open neighbourhoods.

% We may wonder whether an $I$-module with a topology generated as in Theorem \ref{thm: I-module} is such that the map with the formula for addition is continuous. It is not difficult to see that that is the case if and only if the following condition holds \[\forall b\in B,\exists m\in I,\forall a\in S, b+\xi(m,a)=\xi(m,a)+b.\]

\section{Conclusion}

The notion of cartesian spacial fibrous preorder has been introduced as a system $(X,(\leq^{i})_{i\in I},(\partial^{i})_{i\in I})$ indexed over a unitary magma, $I$, and satisfying conditions (C1), (C2) and (C3). Every such structure gives rise to a topological space (Proposition \ref{thm: fibrous to spaces}) and the spaces which are thus obtained were called $I$-cartesian. It has been proven (Theorem~\ref{thm: I-cartesian spaces}) that a space $(X,\tau)$ is $I$-cartesian if and only if there exists a map $\eta\colon{I\times X\to \tau}$ such that $x\in \eta(i,x)$ and  $\eta(i j,x)\subseteq \eta(i,x)\cap \eta(j,x)$ for all $i,j\in I$ and $x\in X$, together with a map $$\gamma\colon{\{(U,x)\mid x\in U\in \tau\}\to I}$$ such that $\eta(\gamma(U,x),x)\subseteq U$. In a sequel to this work we will concentrate our attention on the structure of morphisms between cartesian spacial fibrous preorders thus elevating the characterization of Theorem \ref{thm: I-cartesian spaces} to a categorical equivalence. The structure of metric spaces has been  decomposed as an indexing unitary magma, a preorder, a lax-left-associative Mal'tsev operation, and a linking map (Proposition \ref{thm: metric spaces}). The special case of a normed vector space has been  treated in three different manners:
\begin{enumerate}
\item  In Proposition \ref{thm: a group with a subset gives a fibrous etc} as  a group with a distinguished subset containing the origin and satisfying some conditions (intuitively a convex and bounded neighbourhood of the origin, namely the open ball of radius one) together with a map intuitively measuring how close to the boundary an element is.
\item In Proposition \ref{thm: metric spaces} as a special case of a lax-left-associative Mal'tsev operation with $p(a,x,y)=a-\|y-x\|$.
\item In Theorem \ref{thm: I-module} as an $I$-module  with action $\xi(n,x)=\frac{1}{n} x$, $S=\{a\in B\mid  \|a\|<1\}$ and $\alpha(a)$ such that $\frac{1}{\alpha(a)}\leq 1-\|a\|$.
\end{enumerate}
 A list of examples has been presented as a way to illustrate possible applications. This is a first step for a systematic analysis on the structure of topological spaces.

\end{document}